\documentclass[leqno]{amsart}
\usepackage{amsmath,amssymb,amsthm,mathtools,dsfont,mathrsfs}
\usepackage{multirow}
\DeclarePairedDelimiter\norm{\lvert}{\rvert}
\DeclarePairedDelimiter\inner{\langle}{\rangle}
\usepackage{bbm}

\usepackage{bm}
\usepackage[colorlinks=true,linkcolor=blue,citecolor=blue]{hyperref}
\usepackage{enumitem}
\usepackage{csquotes}

\numberwithin{equation}{section}

\newcounter{intro}

		\newtheorem{introthm}[intro]{Theorem}

		\newtheorem{thm}[equation]{Theorem}
		\newtheorem{lem}[equation]{Lemma}
		\newtheorem{cor}[equation]{Corollary}
		\newtheorem{prop}[equation]{Proposition}
		\newtheorem{conj}[equation]{Conjecture}
\theoremstyle{remark}
		\newtheorem{rem}[equation]{Remark}
\theoremstyle{definition}
		
		\newtheorem{exam}[equation]{Example}

\newcommand{\irr}{\mathrm{Irr}}
\newcommand{\cl}{\mathrm{cl}}
\newcommand{\orb}{\mathrm{orb}}

\title[Supercharacter theories of $C_p\times C_p$]{Toward a Classification of \\the Supercharacter Theories of $C_p\times C_p$}

\author{Shawn T. Burkett}
\address{Department of Mathematical Sciences, Kent State University, Kent,
Ohio 44242, U.S.A.} \email{sburket1@kent.edu}

\author{Mark L. Lewis}
\address{Department of Mathematical Sciences, Kent State University, Kent,
Ohio 44242, U.S.A.} \email{lewis@math.kent.edu}

\date{\today}

\begin{document}
\maketitle

\begin{abstract}
In this paper, we study the superscharacter theories of elementary abelian $p$-groups of order $p^2$.  We show that the supercharacter theories that arise from the direct product construction and the $\ast$-product construction can be obtained from automorphisms.  We also prove that any supercharacter theory of an elementary abelian $p$-group of order $p^2$ that has a nonidentity superclass of size $1$ or a nonprincipal linear supercharacter must come from either a $\ast$-product or a direct product.  Although we are unable to prove results for general primes, we do compute all of the supercharacter theories when $p = 2, 3, 5$, and based on these computations along with particular computations for larger primes, we make several conjectures for a general prime $p$. 
\end{abstract}

\section{Introduction}

A supercharacter theory of a finite group is a somewhat condensed form of its character theory where the conjugacy classes are replaced by certain unions of conjugacy classes and the irreducible characters are replaced by certain pairwise orthogonal characters that are constant on the superclasses. In essence, a supercharacter theory is an approximation of the representation theory that preserves much of the duality exhibited by conjugacy classes and irreducible characters. Supercharacter theory has proven useful in a variety of situations where the full character theory is unable to be described in a useful combinatorial way.

The problem of classifying all supercharacter theories of a given finite group appears to be a difficult problem. For example, in \cite{BLLW17}, the authors saw no way other than to use a computer program to show that $\mathrm{Sp}_6(\mathbb{F}_2)$ has exactly two supercharacter theories. The problem of classifying all supercharacter theories of a {\it family} of finite groups seems likely to be much more difficult. At this time, we know of only a few families of groups for which this has been done; only one of which consists of nonabelian groups. 

In his Ph.D. thesis \cite{AH08}, Hendrickson classified the supercharacter theories of cyclic $p$-groups. It is then explained in the subsequent paper \cite{AH12} that the supercharacter theories of cyclic groups had already been classified by Leung and Man under the guise of {\it Schur rings} (see \cite{cyclicschur,LM98}). The supercharacter theories of the groups $C_2\times C_2\times C_p$ for $p$ a prime were classified in \cite{c2c2cp} (again via Schur rings). The supercharacter theories of the dihedral groups have been classified in a few different ways. Wynn classified the supercharacter theories of dihedral groups in his Ph.D. thesis \cite{CW17}. This is also accomplished by Lamar in his Ph.D. thesis \cite{lamar} (see also the preprint \cite{jonD2n}), where the properties of the lattice of supercharacter theories is also studied. It turns out that the supercharacter theories of the dihedral groups of order twice a prime were classified previously by Wai-Chee \cite{dihedralSchur}, where their Schur rings were studied (supercharacter theories correspond the the {\it central} Schur rings).
Finally we mention that Wynn and the second author classifed the supercharacter theories of Frobenius groups of order $pq$ in \cite{camina} and also reduced the problem of classifying the supercharacter theories of
Frobenius groups and semi-extraspecial groups to classifying the
supercharacter theories of various quotients and subgroups. 

Each of the above classifications involves certain {\it supercharacter theory products}, including $\ast$ and direct products, which will be described explicity in Section~\ref{prelim}. Each of the above classifications also involves supercharacter theories {\it coming from automorphisms} --- those constructed from the action of a group by automorphisms. The purpose of this paper is to study the supercharacter theories of the elementary abelian group $C_p\times C_p$, where $p$ is a prime. The first thing to notice is that a classification of the supercharacter theories of $C_p\times C_p$ would not need to include the above supercharacter theory products.

\begin{introthm}\label{thmA}
	Every supercharacter theory of $C_p\times C_p$ that can be realized as a $\ast$-product or direct product comes from automorphisms.
\end{introthm}

Although we do not give a full classification, we will classify certain types of supercharacter theories. Specifically, we prove the following:

\begin{introthm}\label{thmB}Any supercharacter theory $\mathsf{S}$ of the elementary abelian group $C_p\times C_p$ that has a nonidentity superclass of size one or a nonprincipal linear supercharacter comes from a supercharacter theory ($\ast$ or direct) product. In particular, $\mathsf{S}$ comes from automorphisms. 
\end{introthm}

In Section \ref{partition}, we show that any partition of the nontrivial, proper subgroups of $C_p\times C_p$ gives rise to a supercharacter theory and that these {\it partition supercharacter theories} play a special role in the full lattice of supercharacter theories. In Section \ref{S normal}, we make a strong conjecture regarding the structure of certain types of supercharacter theories that has been supported through computational evidence. Although we do not provide a full classification, we believe this paper to be a good starting point for anyone who desires to do so. We mention that we are able to provide a full classification for the prime $p=2,3$ and $5$. Although we have not been able to complete a classification for $p=7$, all evidence thus far points towards a similar classification as for the smaller primes.

\section{Preliminaries}\label{prelim}

Diaconis and Issacs \cite{ID07} define a {\it supercharacter theory} to be a pair $(\mathcal{X},\mathcal{K})$, where $\mathcal{X}$ is a partition of $\irr(G)$ and $\mathcal{K}$ is a partition of $G$ satisfying the following three conditions:
\begin{itemize}
	\item $\norm{\mathcal{X}}=\mathcal{K}$;
	\item $\{1\}\in\mathcal{K}$;
	\item For every $X\in\mathcal{X}$, there is a character $\xi_X$ whose constituents lie in $X$ that is constant on the parts of $\mathcal{K}$. 
\end{itemize}
	For each $X\in\mathcal{X}$, the character $\xi_X$ is a constant multiple of the character $\sigma_X=\sum_{\psi\in X}\psi(1)\psi$. We let $\mathrm{BCh}(\mathsf{S})=\{\sigma_X:X\in\mathcal{X}\}$ and call its elements the {\it basic} $\mathsf{S}$-{\it characters}.
If $\mathsf{S}=(\mathcal{X},\mathcal{K})$ is a supercharacter theory, we write $\mathcal{K}=\mathrm{Cl}(\mathsf{S})$ and call its elements $\mathsf{S}$-{\it classes}.  The principal character of $G$ is always a basic $\mathsf{S}$-character.  When there is no ambiguity, we may refer to $\mathsf{S}$-classes and basic $\mathsf{S}$-characters as {\it superclasses} and {\it supercharacters}. We will frequently make use of the fact that $\mathsf{S}$-classes and basic $\mathsf{S}$-characters uniquely determine each other \cite[Theorem 2.2 (c)]{ID07}. 

The set $\mathrm{SCT}(G)$ of all supercharacter theories comes equipped with a partial order. Hendrickson \cite{AH12} shows that for any two supercharacter theories $\mathsf{S}$ and $\mathsf{T}$, every $\mathsf{T}$-class is a union of $\mathsf{S}$-classes if and only if every basic $\mathsf{T}$-character is a sum of basic $\mathsf{S}$-characters. In this event, we write $\mathsf{S}\preccurlyeq\mathsf{T}$. We say that $\mathsf{S}$ is finer than $\mathsf{T}$ or that $\mathsf{T}$ is coarser than $\mathsf{S}$. Since $\mathrm{SCT}(G)$ has a partial order and a maximal (and minimal) element, it is actually a lattice. The join operation $\vee$ on $\mathrm{SCT}(G)$ is very well-behaved and is inherited from the join operation on the set of partitions of $G$ under the refinement order, which we will also denote by $\vee$. If $\mathsf{S}$ and $\mathsf{T}$ are supercharacter theories of $G$, then the superclasses of $\mathsf{S}\vee\mathsf{T}$ is just the mutual coarsening of the partitions $\mathrm{Cl}(\mathsf{S})$ and $\mathrm{Cl}(\mathsf{T})$; i.e. $\mathrm{Cl}(\mathsf{S})\vee\mathrm{Cl}(\mathsf{T})$. However, the meet operation $\wedge$ on $\mathrm{SCT}(G)$ is poorly behaved and difficult to compute. In particular, the equation $\mathrm{Cl}(\mathsf{S}\wedge\mathsf{T})=\mathrm{Cl}(\mathsf{S})\wedge\mathrm{Cl}(\mathsf{T})$ holds only sporadically. One example where this equality does hold will be discussed later in this section (see Lemma~\ref{direct prod}).

Every finite group has two {\it trivial} supercharacter theories. The first, which we denote by $\mathsf{m}(G)$, is the supercharacter theory with superclasses the usual conjugacy classes of $G$. The supercharacters of $\mathsf{m}(G)$ are exactly the irreducible characters of $G$ (multiplied by their degrees). This is the finest supercharacter theory of $G$ under the partial order discussed in the previous paragraph (i.e. $\mathsf{m}(G)\preccurlyeq\mathsf{S}$ for every supercharacter theory $\mathsf{S}$ of $G$). There is also a coarsest supercharacter theory of $G$ for the partial ordering of the previous paragraph, denoted by $\mathsf{M}(G)$ (i.e. $\mathsf{S}\preccurlyeq\mathsf{M}(G)$ for every supercharacter theory $\mathsf{S}$ of $G$). The $\mathsf{M}(G)$-classes are just $\{1\}$ and $G\setminus\{1\}$ and the basic $\mathsf{M}(G)$-characters are $\mathbbm{1}$ and $\rho_G-\mathbbm{1}$, where $\mathbbm{1}$ is the principal character and $\rho_G$ is the regular character of $G$.

Supercharacter theories can arise in many different (often mysterious) ways. One of the more well-known ways comes from actions by automorphisms. If $A\le\mathrm{Aut}(G)$, then $A$ acts on $\irr(G)$ via $\chi^a(g)=\chi(g^{a^{-1}})$ for $a\in A$, $\chi\in\irr(G)$ and $g\in G$. Then Brauer's Permutation Lemma (see \cite[Theorem 6.32]{MI76}, for example) can be used to show that the orbits of $G$ and $\irr(G)$ under the action of $A$ yield a supercharacter theory. In this case, we say that $\mathsf{S}$ {\it comes from $A$} or {\it comes from automorphisms}. An important aspect of the Leung--Man classification \cite{cyclicschur,LM98} (or Hendrickson's \cite{AH08}) is that every supercharacter theory of a cyclic group of prime order comes from automorphisms. This fact will be used extensively later without reference. 

Just as every normal subgroup is determined by the conjugacy classes of $G$ and by the irreducible characters, there is a distinguished set of normal subgroups determined by $\mathsf{S}$. Any subgroup $N$ that is a union of $\mathsf{S}$-classes is called $\mathsf{S}$-normal. In this situation, we write $N\lhd_{\mathsf{S}}G$. It is not difficult to show that $N$ is the intersection of the kernels of those $\chi\in\mathrm{BCh}(\mathsf{S})$ that satisfy $N\le\ker(\chi)$. In fact this is another way to classify $\mathsf{S}$-normal subgroups \cite{EM11}.

Whenever $N$ is $\mathsf{S}$-normal, Hendrickson \cite{AH12} showed that $\mathsf{S}$ gives rise to a supercharacter theory $\mathsf{S}_N$ of $N$ and $\mathsf{S}^{G/N}$ of $G/N$. The $\mathsf{S}_N$-classes are just the $\mathsf{S}$-classes contained in $N$ and the basic $\mathsf{S}_N$-characters are the restrictions of the basic $\mathsf{S}$-characters, up to a constant. Moreover, it is a result of the first author (see \cite[Theorem 1.1.2]{SB18} or \cite[Theorem A]{SB18nil}) that if $\chi\in\mathrm{BCh}(\mathsf{S})$ and $\psi$ is the basic-$\mathsf{S_N}$-character lying under $\chi$, then $\chi(1)/\psi(1)$ is an integer. The $\mathsf{S}^{G/N}$-classes are the images of the $\mathsf{S}$-classes under the canonical projection $G\to G/N$ and the basic $\mathsf{S}^{G/N}$-characters can be identified with the basic $\mathsf{S}$-characters with $N$ contained in their kernel. These constructions are compatible in the sense that $(\mathsf{S}_N)^{N/M}=(\mathsf{S}^{G/M})_{N/M}$. As such, we simply write $\mathsf{S}_{N/M}$ in this situation. In \cite{SBJH18}, the first author shows that these constructions respect the lattice structure of the set of $\mathsf{S}$-normal subgroups. In particular, if $H$ and $N$ are any $\mathsf{S}$-normal subgroups, then the images of the superclasses of $\mathsf{S}_{H/(H\cap N)}$ under the canoncial isomorphism $H/(H\cap N)\to HN/N$ are exactly the $\mathsf{S}_{HN/N}$-classes \cite[Theorem A]{SBJH18}.

Hendrickson also used these constructions to define supercharacter theories of the full group. Given any supercharacter theory $\mathsf{U}$ of a normal subgroup $N$ of $G$ whose superclasses are fixed (set-wise) under the conjugation action of $G$ and a supercharacter theory $\mathsf{V}$ of $G/N$, Hendrickson defines the $\ast$-{\it product} $\mathsf{U}\ast\mathsf{V}$ as follows. The supercharacters of $\mathsf{U}\ast\mathsf{V}$ that have $N$ in their kernel can be naturaly identified with the supercharacters in $\mathrm{BCh}(\mathsf{V})$ and those that do not have $N$ in their kernel are just induced from nonprincipal members of $\mathrm{BCh}(\mathsf{U})$. The superclasses of $\mathsf{U}\ast\mathsf{V}$ contained in $N$ are the superclasses of $\mathsf{U}$, and the superclasses of $\mathsf{U}\ast\mathsf{V}$ lying outside of $N$ are the full preimages of the nonidentity superclasses of $\mathsf{V}$ under the canonical projection $G\to G/N$. If $\mathsf{S}$ is a supercharacter theory of $G$ and $N\lhd _{\mathsf{S}}G$, then $\mathsf{S}\preccurlyeq\mathsf{S}_N\ast\mathsf{S}_{G/N}$, with equality if and only if every $\mathsf{S}$-class lying outside of $N$ is a union of $N$-cosets. 

Another characterization that appears in \cite{SBMLsctProducts} is the following. Let $N$ be $\mathsf{S}$-normal. Then $\mathsf{S}$ is a $\ast$-product over $N$ if and only if every $\chi\in\mathrm{BCh}(\mathsf{S})$ satisfying $N\nleq\ker(\chi)$ vanishes on $G\setminus N$. One direction of this result follows easily from the next lemma about basic $\mathsf{S}$-characters vanishing off $\mathsf{S}$-normal subgroups.

\begin{lem}\label{induce}
	Let $\mathsf{S}$ be a supercharacter theory of $G$ and let $\chi\in\mathrm{BCh}(\mathsf{S})$. Assume that $\chi$ vanishes on $G\setminus N$, where $N$ is $\mathsf{S}$-normal. Then $\chi=\psi^G$ for some basic $\mathsf{S}_N$-character $\psi$. 
\end{lem}

\begin{proof}
	Since $\mathsf{S}\preccurlyeq\mathsf{S}_N\ast\mathsf{S}_{G/N}$, $\psi^G$ is a sum of distinct basic $\mathsf{S}$-characters. Let $\xi$ be one such basic $\mathsf{S}$-character, and note that $\chi_N=\frac{\chi(1)}{\psi(1)}\psi$. Then
	\begin{align*}
		\inner{\psi^G,\xi}
		&=\frac{1}{\norm{G}}\sum_{g\in G}\psi^G(g)\overline{\xi(g)}=\frac{1}{\norm{G}}\sum_{g\in N}\psi^G(g)\overline{\xi(g)}\\
		&=\frac{1}{\norm{N}}\sum_{g\in N}\psi(g)\overline{\xi(g)}=\frac{\psi(1)}{\norm{N}\chi(1)}\sum_{g\in N}\chi_N(g)\overline{\xi(g)}\\
		&=\frac{\psi(1)}{\norm{N}\chi(1)}\sum_{g\in G}\chi(g)\overline{\xi(g)}=\frac{\norm{G:N}}{\chi(1)/\psi(1)}\inner{\chi,\xi}=\norm{G:N}\psi(1)\delta_{\chi,\xi}.
	\end{align*}
	The result easily follows.
\end{proof}

Also defined in \cite{AH12} is the {\it direct product} of supercharacter theories. Given a supercharacter theory $\mathsf{E}$ of a group $H$ and a supercharacter theory $\mathsf{F}$ of a group $K$, the supercharacter theory $\mathsf{E}\times\mathsf{F}$ of $H\times K$. This supercharacter theory is defined by $\mathrm{Cl}(\mathsf{E}\times\mathsf{F})=\{K\times L:K\in\mathrm{Cl}(\mathsf{E}),\ L\in\mathrm{Cl}(\mathsf{F})\}$ and $\mathrm{BCh}(\mathsf{S})=\{\chi\times\xi:\chi\in\mathrm{BCh}(\mathsf{E}),\ \xi\in\mathrm{BCh}(\mathsf{S})\}$. The direct product supercharacter theory is intimately related to the $\ast$-product, as this next result illustrates.

\begin{lem}\label{direct prod}
Let $G=H\times N$, let $\mathsf{S}\in\mathrm{SCT}(H)$, and let $\mathsf{T}\in\mathrm{SCT}(N)$. Let $\varphi_1:H\to G/N$ and $\varphi_2:N\to G/H$ be the projections. Let $\tilde{\mathsf{S}}=\varphi_1(\mathsf{S})\in\mathrm{SCT}(G/N)$, and let $\tilde{\mathsf{T}}=\varphi_2(\mathsf{T})\in\mathrm{SCT}(G/H)$. Write $\mathsf{U}=\mathsf{S}\ast\tilde{T}$ and $\mathsf{V}=\mathsf{T}\ast\tilde{S}$. Then $\mathrm{Cl}(\mathsf{S}\times\mathsf{T})=\mathrm{Cl}(\mathsf{U})\wedge\mathrm{Cl}(\mathsf{V})$. In particular, $\mathsf{S}\times\mathsf{T}$ is equal to $\mathsf{U}\wedge\mathsf{V}$.
\end{lem}
\begin{proof}
We have
\[\mathrm{Cl}(\mathsf{U})=\bigcup_{K\in\mathrm{Cl}(\mathsf{S})}\{K\times\{1\},K\times(N\setminus\{1\})\}\]
and
\[\mathrm{Cl}(\mathsf{V})=\bigcup_{L\in\mathrm{Cl}(\mathsf{T})}\{\{1\}\times L,(H\setminus\{1\})\times L\}.\]
The mutual refinement of these partitions is exactly 
\[\mathcal{K}=\{K\times L:K\in\mathrm{Cl}(\mathsf{S}),\ L\in\mathrm{Cl}(\mathsf{T})\}.\]
Since $\mathcal{K}$ is the set of superclasses of a supercharacter theory of $G$, and $\mathcal{K}$ is the coarsest partition of $G$ finer than both $\mathrm{Cl}(\mathsf{U})$ and $\mathrm{Cl}(\mathsf{T})$, it follows that $\mathcal{K}=\mathrm{Cl}(\mathsf{U}\wedge\mathsf{V})$, which means $\mathsf{U}\wedge\mathsf{V}=\mathsf{S}\times\mathsf{T}$.
\end{proof}

Recall that $\mathsf{S}\preccurlyeq\mathsf{S}_N\ast\mathsf{S}_{G/N}$ whenever $N$ is an $\mathsf{S}$-normal subgroup of $G$. Thus, as an immediate corollary of Lemma~\ref{direct prod}, we deduce the following.

\begin{cor}\label{direct}
Let $G=H\times N$ and suppose $\mathsf{S}$ is a supercharacter theory of $G$ in which both $H$ and $N$ are $\mathsf{S}$-normal. Then $\mathsf{S}\preccurlyeq\mathsf{S}_H\times\mathsf{S}_N$, with equality if and only if $\norm{\mathsf{S}}=\norm{\mathsf{S}_H}\cdot\norm{\mathsf{S}_N}$.
\end{cor}
\begin{proof}
Using the notation in the statement of previous result, we have $\widetilde{\mathsf{S}_H}=\mathsf{S}_{G/N}$ and $\widetilde{\mathsf{S}_N}=\mathsf{S}_{G/H}$ \cite[Theorem A]{SBJH18}. Since 
\[\mathsf{S}\preccurlyeq\mathsf{S}_H\ast\mathsf{S}_{G/H}=\mathsf{S}_H\ast\widetilde{\mathsf{S}_N}\]
and
\[\mathsf{S}\preccurlyeq\mathsf{S}_N\ast\mathsf{S}_{G/N}=\mathsf{S}_N\ast\widetilde{\mathsf{S}_H},\]
we have
\[\mathsf{S}\preccurlyeq\bigl(\mathsf{S}_H\ast\widetilde{\mathsf{S}_N}\bigr)\wedge\bigl(\mathsf{S}_N\ast\widetilde{\mathsf{S}_H}\bigr)=\mathsf{S}_H\times\mathsf{S}_N.\]
Since $\mathsf{S}\preccurlyeq\mathsf{S}_H\times\mathsf{S}_N$ and $\norm{\mathsf{S}_H\times\mathsf{S}_N}=\norm{\mathsf{S}_H}\cdot\norm{\mathsf{S}_N}$, the result follows.
\end{proof}

We mention one more construction we will need, also due to Hendrickson. If $G$ is an abelian group, then $\irr(G)$ forms a group under the pointwise product. There is a natural isomorphism $G\to\irr(\irr(G))$ sending $g\in G$ to $\tilde{g}\in\irr(\irr(G))$ defined by $\tilde{g}(\chi)=\chi(g)$. If $\mathsf{S}$ is a supercharacter theory of $G$, then $\check{\mathsf{S}}$ is a supercharacter theory of $\irr(G)$, where $\mathrm{Cl}(\check{\mathsf{S}})=\mathrm{BCh}(\mathsf{S})$ and $\mathrm{BCh}(\check{\mathsf{S}})=\{\{\tilde{g}:g\in K\}:K\in\mathrm{Cl}(\mathsf{S})\}$ \cite[Theorem 5.3]{AH08}. This {\it duality} construction will be used to simplify some arguments in the proof of Theorem~\ref{thmB}.


\section{Central elements and commutators}

Let $\mathsf{S}$ be a supercharacter theory of $G$. In \cite{SB18nil}, the first author discusses two important subgroups of $G$ associated to $\mathsf{S}$. The first of these subgroups is an analog of the center of a group and consists of the superclasses of size one. We denote this subgroup by $Z(\mathsf{S})$. The fact that $Z(\mathsf{S})$ is a ($\mathsf{S}$-normal) subgroup follows easily from \cite[Corollary 2.3]{ID07} and a proof appears in \cite{AH08}. Another consequence of \cite[Corollary 2.3]{ID07} appearing in \cite{AH08} is that $\cl_{\mathsf{S}}(g)z=\cl_{\mathsf{S}}(gz)$ for any $z\in Z(\mathsf{S})$. Using this fact, as well as a consequence of \cite[Theorem A]{SB18nil}, we prove the following lemma that will be used in the proof of Theorem~\ref{thmB}.

\begin{lem}\label{class}
Let $\mathsf{S}$ be a supercharacter theory of $G$ and write $Z=Z(\mathsf{S})$. If $gz\notin\cl_{\mathsf{S}}(g)$ for any $z\in Z$, then $\norm{\cl_{\mathsf{S}_{G/Z}}(gZ)}=\norm{\cl_{\mathsf{S}}(g)}$.
\end{lem}

\begin{proof}
Let $h\in\cl_{\mathsf{S}}(g)$. Then $h=gg^{-1}h\in\cl_{\mathsf{S}}(g)$, so $g^{-1}h\notin Z$. So the map $\cl_{\mathsf{S}}(g)\to\cl_{\mathsf{S}_{G/Z}}(gZ)$ is injective. Since $\norm{\cl_{\mathsf{S}_{G/Z}}(gZ)}$ divides $\norm{\cl_{\mathsf{S}}(g)}$, the result follows.
\end{proof}

It turns out that many analogs of classical results about the center of the group exist for $Z(\mathsf{S})$ (see \cite{SB18nil} for more details). Among these is the next result, which is a generalization of a well-known fact about ordinary complex characters (e.g., see \cite[Corollary 2.30]{MI76}).

\begin{lem}\label{centerequality}
	Let $\chi$ be a basic $\mathsf{S}$-character of $G$ and write $Z=Z(\mathsf{S})$. Then $\chi(1)\le\norm{G:Z(\mathsf{S})}$, with equality if and only if $\chi$ vanishes on $G\setminus Z(\mathsf{S})$.
\end{lem}

\begin{proof}
	 Since $\mathsf{S}_{Z(\mathsf{S})}$ is the finest supercharacter theory, the restriction $\chi_{Z(\mathsf{S})}$ is a multiple of some linear character $\lambda$. So $\inner{\chi_{Z(\mathsf{S})},\chi_{Z(\mathsf{S})}}=\chi(1)^2$. On the other hand, $\inner{\chi_{Z(\mathsf{S})},\chi_{Z(\mathsf{S})}}\le\norm{G:Z(\mathsf{S})}\inner{\chi,\chi}=\norm{G:Z(\mathsf{S})}\chi(1)$, with equality if and only if $\chi$ vanishes on $G\setminus Z(\mathsf{S})$. The result follows.
\end{proof}

We now discuss an analog of the commutator subgroup of $G$. Note that one may write $[G,G]=\inner{g^{-1}k:k\in\cl_G(g)}$. Using this description, it is natural to consider the subgroup $\inner{g^{-1}k:k\in\cl_{\mathsf{S}}(g)}$, which we denote by $[G,\mathsf{S}]$. It turns out this subgroup is always $\mathsf{S}$-normal \cite[Proposition 3.7]{SB18nil}. Moreover, $[G,\mathsf{S}]$ provides information of the structure of the basic $\mathsf{S}$-characters. Mosty notably, a basic $\mathsf{S}$-character $\chi$ is linear if and only if $[G,\mathsf{S}]\le\ker(\chi)$ \cite[Proposition 3.11]{SB18nil}. 

As stated above, if $\mathsf{S}$ is a supercharacter theory of $G$, $N$ is $\mathsf{S}$-normal, $\chi$ is a basic $\mathsf{S}$-character and $\psi$ is a basic $\mathsf{S}_N$-character satisfying $\inner{\chi_N,\psi}>0$, then $\psi(1)$ divides $\chi(1)$. The next result, which is \cite[Proposition 3.13]{SB18nil}, shows this can be strengthened in certain situations, a fact that will be useful later.

\begin{lem}\label{actiondivide}
Let $N$ be an $\mathsf{S}$-normal subgroup satisfying $[G,\mathsf{S}]\le N$. Let $\chi$ be a basic $\mathsf{S}$-character that does not contain $N$ in its kernel, and suppose that $\psi\in\mathrm{BCh}(\mathsf{S}_N)$ satisfies $\inner{\chi_N,\psi}>0$. Then $\chi(1)/\psi(1)$ divides $\norm{G:N}$. 
\end{lem}

\begin{proof}
Since $[G,\mathsf{S}]\le N$, $\Lambda=\mathrm{Ch}(\mathsf{S}/N)$ acts on $\mathrm{BCh}(\mathsf{S})$ in the obvious way. Consider the set $C=\{\psi\lambda:\ \psi\in X,\ \lambda\in\Lambda\}$, where $X=\irr(\chi)$. On the one hand, $C$ is exactly the set of constituents of $\psi^G$. By \cite[Lemma 3.4]{AH08}, we conclude that
\[\sigma_C(1)=\psi^G(1)=\norm{G:N}\psi(1).\]
On the other hand, we have $C=\bigcup_{\lambda\in \Lambda}X^\lambda$. Since $\irr(\chi^\lambda)\cap \irr(\chi)=\varnothing$ whenever $\chi^\lambda\neq \chi$ and $\chi^\lambda(1)=\chi(1)$ for each $\lambda\in\Lambda$, we have
\[\sigma_C(1)=\norm{\mathrm{orb}_{\Lambda}(\chi)}\chi(1).\]
Thus, we have
\[\chi(1)=\frac{\norm{G:N}\psi(1)}{\norm{\mathrm{orb}_{\Lambda}(\chi)}}=\norm{\mathrm{Stab}_{\Lambda}(\chi)}\psi(1).\]
The result follows as $\norm{\mathrm{Stab}_{\Lambda}(\chi)}$ divides $\norm{G:N}$.
\end{proof}

\begin{rem}
	If $\chi\in\irr(G)$ and $\psi\in\irr(N)$ lies under $\chi$, then $\chi(1)/\psi(1)$ is known to divide $\norm{G:N}$ (for a proof, see \cite[Corollary 11.29]{MI76}). In the case $\mathsf{S}=\mathsf{m}(G)$, Lemma~\ref{actiondivide} is saying something a little stronger. In this case, a basic $\mathsf{S}$-character has the form $\chi(1)\chi$ for some $\chi\in\irr(G)$. If $\psi\in\irr(N)$ lies under $\chi$, then a basic $\mathsf{S}_N$ character has degree $\norm{G:I_G(\psi)}$, where $I_G(\psi)$ is the inertia group of $\psi$ in $G$. Therefore, in this case Lemma~\ref{actiondivide} is saying that if $G/N$ is abelian, $(\chi(1)/\psi(1))^2$ divides $\norm{G:N}\norm{G:I_G(\psi)}$.
\end{rem}

\section{Proofs}

In this section, we prove the main results of the paper. For the remainder of the paper, $p$ is an odd prime and $G$ is the abelian group of order $p^2$ and exponent $p$. 

Our first result shows that every $\ast$-product and direct product supercharacter theory of $G$ comes from automorphisms. Note that this includes Theorem~\ref{thmA}.

\begin{lem}\label{autoproducts}
	Let $N\ne M$ be a nontrivial, proper subgroups of $G$. Let $\varphi:M\to G/N$ be the canonical isomorphism $x\mapsto xN$. Let $\mathsf{U}$ be a supercharacter theory of $N$, and let $\mathsf{V}$ be a supercharacter theory of $M$.  The following hold:
	\begin{enumerate}[label={\bf(\arabic*)}]
		\item There exists $A\le \mathrm{Aut}(G)$ such that $\mathsf{U}\ast\varphi(\mathsf{V})$ comes from $A$;
		\item There exists $B\le \mathrm{Aut}(G)$ such that $\mathsf{U}\times\mathsf{V}$ comes from $B$.
	\end{enumerate}
\end{lem}

\begin{proof}
	Write $N=\inner{x}$ and $M=\inner{y}$. Since $N$ and $M$ are cyclic of prime order, there exist integers $m_1$, $m_2$ such that $\mathsf{U}$ comes from the automorphism $\sigma:N\to N$ defined by $x^\sigma=x^{m_1}$ and $\mathsf{U}$ comes from the automorphism $\tau:M\to M$ defined by $y^\tau=y^{m_2}$. 
	
	First we prove (1). Let $\mathsf{S}=\mathsf{U}\ast\varphi(\mathsf{V})$. Then the $\mathsf{S}$-classes contained in $N$ are the orbits of $\inner{\sigma}$ on $N$. The $\mathsf{S}$-classes lying outside of $N$ are the full preimages of the orbits of $\inner{\tau}$ on $M$ under the projection $G\to G/N$. Extend $\sigma$ to an automorphism $\tilde{\sigma}$ of $G$ by setting $y^{\tilde{\sigma}}=y$. For each $1\le k\le p-1$, define the automorphism $\tau_k$ of $G$ by defining $(x^iy^j)^{\tau_k}=x^{i+jk}y^{jm_2}$. Let $A=\inner{\tilde{\sigma},\tau_1,\tau_2,\dotsc,\tau_{p-1}}$. Then $N$ is $A$-invariant and the orbits of $A$ on $N$ are exactly the orbits of $\inner{\sigma}$ on $N$. Thus $\orb_A(g)=\cl_{\mathsf{S}}(g)$ if $g\in N$. Observe that $\{1,y,\dotsc,y^{p-1}\}$ is a transversal for $N$ in $G$. For each $1\le j,k\le p-1$, observe that $(y^j)^{\tau_k}=y^{jm_2}x^{jk}$. As $k$ ranges over the set $\{1,2,\dotsc,p-1\}$, so does $jk$. It follows that $\orb_A(y^j)=\{hn:h\in\orb_{\inner{\tau}}(y^j), n\in N\}$. For each $g\in G$, we may write $g$ uniquely as $g=g_Ng_M$ where $g_N\in N$ and $g_M$. Let $g\in G$. From the arguments above, we see that $\orb_A(g)=\{hn:h\in\orb_{\inner{\tau}}(g_M), n\in N\}$, which is exactly $\cl_{\mathsf{S}}(g)$. This completes the proof of (1).
	
	Now we show (2). Let $\mathsf{D}=\mathsf{U}\times\mathsf{V}$. Extend $\sigma$ to an automorphism $\tilde{\sigma}$ of $G$ by setting $y^{\tilde{\sigma}}=y$, and extend $\tau$ to an automorphism $\tilde{\tau}$ of $G$ by setting $x^{\tilde{\tau}}=x$. Let $B=\inner{\tilde{\sigma},\tilde{\tau}}$. Then $N$ and $M$ are both $B$-invariant. If $g\in N\cup M$, then it is easy to see that $\orb_B(g)=\cl_{\mathsf{D}}(g)$. If $g\not\in N\cup M$, then 
	\[\orb_B(g)=\bigcup_{i=1}^{d_2}\bigl\{g_N^{m_1}g_M^{m_2^i},g_N^{m_1^2}g_M^{m_2^i},\dotsc,g_N^{m_1^{d_1-1}}g_M^{m_2^i}\bigr\}\]
	 where $d_i$ is the order of $m_i$ modulo $p$. Thus $\orb_B(g)=\orb_{\inner{\sigma}}(g_N)\times\orb_{\inner{\tau}}(g_M)=\cl_{\mathsf{S}}(g)$. This completes the proof of (2).
\end{proof}

We may now give a more precise statement and proof of Theorem~\ref{thmB}. We first remark that having a nonidentity superclass of size $1$ is equivalent to the condition $Z(\mathsf{S}) > 1$ and having a nonprincipal linear supercharacter
is equivalent to the condition $[G,\mathsf{S}] < G$.

\begin{thm}\label{main actual}
Let $\mathsf{S}$ be a supercharacter theory of $G$. Assume that $[G,\mathsf{S}]<G$ or $Z(\mathsf{S})>1$. One of the following holds:
	\begin{enumerate}[label={\bf(\arabic*)}]
		\item $\mathsf{S}$ is a $\ast$-product over $[G,\mathsf{S}]$.
		\item $\mathsf{S}$ is a $\ast$-product over $Z(\mathsf{S})$.
		\item $\mathsf{S}$ is a direct product over $[G,\mathsf{S}]$ and $Z(\mathsf{S})$.
	\end{enumerate}
In particular, $\mathsf{S}$ comes from automorphisms. 
\end{thm}

\begin{proof}
First observe that $[G,\mathsf{S}]=1$ if and only if $Z(\mathsf{S})=G$, and the result holds trivially in this case. Thus it suffices to assume that $[G,\mathsf{S}]$ or $Z(\mathsf{S})$ is nontrivial and proper. We can make another reduction by considering $G^\ast=\irr(G)$---the dual of $G$---and the dual supercharacter theory $\check{S}$ of $\mathsf{S}$. If $\mathsf{S}=(\mathcal{X},\mathcal{K})$, then $Z(\mathsf{S})$ and $\mathrm{BCh}(G/[G,\mathsf{S}])$ are comprised of the parts of $\mathcal{K}$ and $\mathcal{X}$ of size 1, respectively. From this description, it is clear that $Z(\mathsf{S})^\ast=G^\ast/[G^\ast,\check{S}]$ and $Z(\check{\mathsf{S}})=(G/[G,\mathsf{S}])^\ast$. Thus it suffices to prove that the result holds in the case that  $1<[G,\mathsf{S}]<G$, which we now assume.

We first assume additionally that $[G,\mathsf{S}]$ is the unique $\mathsf{S}$-normal subgroup of index $p$. Then either ${Z}(\mathsf{S})=1$ or ${Z}(\mathsf{S})=[G,\mathsf{S}]$. Assume that ${Z}(\mathsf{S})=[G,\mathsf{S}]$. Let $\chi$ be an $\mathsf{S}$-character without $[G,\mathsf{S}]$ in its kernel. Then $\chi(1)=\norm{G:{Z}(\mathsf{S})}=p$. By Lemma~\ref{centerequality}, we deduce that $\chi$ vanishes on $G\setminus {Z}(\mathsf{S})$. Hence we see from \cite[Theorem 2]{SBMLsctProducts} that $\mathsf{S}$ is a $\ast$-product over $[G,\mathsf{S}]$. So we now assume that ${Z}(\mathsf{S})=1$. Let $\chi\in\mathrm{BCh}(\mathsf{S}\mid[G,\mathsf{S}])$, and let $\psi\in\mathrm{BCh}([G,\mathsf{S}])$ lie under $\chi$. If every element of $\irr(G/[G,\mathsf{S}])$ fixes $\chi$, then $\chi$ vanishes on $G\setminus[G,\mathsf{S}]$ and so $\chi=\psi^G$ by Lemma~\ref{induce}. Let $1\le j\le p-1$. Then $\chi^j\in\mathrm{BCh}(\mathsf{S})$ and $\chi^j=(\psi^j)^G$, where here $\alpha^j(g)=\alpha(g^j)$. As $j$ ranges over all $j$, $\psi^j$ ranges over all nonprincipal basic $\mathsf{S}_{[G,\mathsf{S}]}$-characters. Thus we see that every $\chi\in\mathrm{BCh}(\mathsf{S}\mid[G,\mathsf{S}])$ vanishes on $G\setminus[G,\mathsf{S}]$. Hence $\mathsf{S}$ is a $\ast$-product over $[G,\mathsf{S}]$ in this case as well. So assume that this is not the case. We instead consider the dual situation in $G^\ast=\irr(G)$. Then $Z={Z}(\check{\mathsf{S}})$ has order $p$, $[G^\ast,\check{\mathsf{S}}]=G^\ast$ and that $\cl_{\check{\mathsf{S}}}(g)$ is not fixed by the action of $Z$ for any $g\in G^\ast\setminus Z$. Let $h\in G^\ast\setminus Z$. Then $\inner{h}$ gives a transversal for $Z$ in $G^\ast$. Since $G^\ast/Z\simeq C_p$, there exists $1\le m\le p-1$ such that $\cl_{\mathsf{\check{S}}_{G/Z}}(hZ)=\{hZ,h^mZ,\dotsc,h^{m^{d-1}}Z\}$, where $d=\mathrm{ord}_p(m)$. Also, since $\norm{\cl_{\check{\mathsf{S}}}(h)}=\norm{\cl_{\check{\mathsf{S}}_{G/Z}}(hZ)}$ by Lemma~\ref{class}, there exist $z_1,z_2,\dotsc,z_{d-1}\in{Z}(\check{\mathsf{S}})$ such that $\cl_{\check{\mathsf{S}}}(h)=\{h,h^mz_1,h^{m^2}z_2,\dotsc,h^{m^{d-1}}z_{d-1}\}$. So 
\[\cl_{\check{\mathsf{S}}}(h^m)=\cl_{\mathsf{S}}(h)^m=\{h^m,h^{m^2}z_1^m,h^{m^3}z_2^m,\dotsc,hz_{d-1}^m\},\]
from which it follows that $\cl_{\check{\mathsf{S}}}(h^m)z_1=\cl_{\check{\mathsf{S}}}(h)$. Since $\cl_{\check{\mathsf{S}}}(h)$ is not fixed by multiplication by any element of $Z$, this implies that $z_2=z_1^{m+1}$. Similarly $z_3=z_2^mz_1=z_1^{m^2+m+1}$. Continuing this way, we deduce that $z_i =z^{1+m+m^2+\dotsb+m^{i-1}}=z_1^{(m^i-1)/(m-1)}$ for each $1\le i\le d-1$, where $z=z_1$. Thus we see that
$h^{-1}k\in\inner{h^{m-1}z}$ for every $k\in\cl_{\check{\mathsf{S}}}(h)$. Similarly, for every $w\in Z$ and $k\in\cl_{\check{\mathsf{S}}}(hw)$, $(hw)^{-1}k\in\inner{h^{m-1}z}$. Since every element of $G^\ast$ has the form $h^jw$ for some integer $j$ and $w\in Z$, it follows that $g^{-1}k\in\inner{h^{m-1}z}$ for every $g\in G^\ast$ and $k\in\cl_{\check{\mathsf{S}}}(g)$. This implies that $[G^\ast,\check{\mathsf{S}}]=\inner{h^{m-1}z}<G^\ast$, a contradiction. 

We may now assume that there is another $\mathsf{S}$-normal subgroup, say $N$. Since $[G,\mathsf{S}]<G$, $[G,\mathsf{S}]\ne Z(\mathsf{S})$. So $N\cap [G,\mathsf{S}]=1$, which implies $N={Z}(\mathsf{S})$. Let $\chi\in\mathrm{BCh}(\mathsf{S})$ satisfy $[G,\mathsf{S}]\nleq\ker(\chi)$. Let $\psi\in\mathrm{BCh}(\mathsf{S}_{[G,\mathsf{S}]})$ lie under $\chi$. Then $\chi(1)/\psi(1)\in\{1,p\}$ by Lemma~\ref{actiondivide}. Assume that $\chi(1)=p\psi(1)$. Since $\chi(1)\le\norm{G:Z(\mathsf{S})}=p$ by Lemma~\ref{centerequality}, we deduce that $\psi(1)=1$. Since $[G,\mathsf{S}]\nleq\ker(\chi)$, $\psi$ is nonprincipal and thus $[[G,\mathsf{S}],\mathsf{S}]=1$. We conclude that $\mathsf{S}_{[G,\mathsf{S}]}=\mathsf{m}([G,\mathsf{S}])$, which contradicts the fact that $[G,\mathsf{S}]\ne Z(\mathsf{S})$. So $\chi(1)=\psi(1)$ and so $\psi^G$ is the sum of $p$ distinct $\mathsf{S}$-characters. If $\mathrm{BCh}(\mathsf{S}/[G,\mathsf{S}])$ is the set of basic $\mathsf{S}$-characters with $[G,\mathsf{S}]$ in their kernel and $\mathrm{BCh}(\mathsf{S}\mid[G,\mathsf{S}])$ is the set of those without $[G,\mathsf{S}]$ in their kernel, then by the above argument we see
\begin{align*}\norm{\mathrm{BCh}(\mathsf{S})}&=\norm*{\mathrm{BCh}(\mathsf{S}/[G,\mathsf{S}])}+p\norm*{\mathrm{BCh}(\mathsf{S}\mid[G,\mathsf{S}])}\\
&=p+p\norm*{\mathrm{BCh}(\mathsf{S}\mid[G,\mathsf{S}])}\\
&=p\norm*{\mathrm{BCh}(\mathsf{S}_{[G,\mathsf{S}]})}=\norm*{\mathsf{S}_{{Z}(\mathsf{S})}}\cdot\norm*{\mathsf{S}_{[G,\mathsf{S}]}}.
\end{align*}
It follows from Corollary~\ref{direct} that $\mathsf{S}=\mathsf{S}_{{Z}(\mathsf{S})}\times\mathsf{S}_{[G,\mathsf{S}]}$.

The final statement is a consequence of Lemma~\ref{autoproducts}
\end{proof}

\section{Partition supercharacter theories} \label{partition}

We now describe a type of supercharacter theory that is (essentially) unique to elementary abelian groups of rank two. As in the previous section, $p$ is a prime and $G$ is the elementary abelian group $C_p\times C_p$.

\begin{lem}\label{partitionsct}
Let $G$ and let $H_1,H_2,\dotsc,H_{p+1}$ be the nontrivial, proper subgroups of $G$. For every subset $I\subseteq\{1,2,\dotsc,p+1\}$, define $N_I=\bigcup_{i\in I}(H_i\setminus1)$. For every partition $\mathcal{P}$ of $\{1,2,\dotsc,p+1\}$, the partition $\{N_I:I\in\mathcal{P}\}$
of $G\setminus 1$ gives the nonidentity superclasses for a supercharacter theory $\mathsf{S}_{\mathcal{P}}$ of $G$.
\end{lem}

\begin{proof}
	To prove this, it suffices to show that there exist nonnegative integers $a_{I,J,1}$ and $a_{I,J,L}$ such that 
	\[\widehat{N_I}\widehat{N_J}=a_{I,J,1}\cdot1+\sum_{L\in\mathcal{P}}a_{I,J,L}\widehat{N_L}\]
	for every $I,J\in\mathcal{P}$. To that end, let $I,J\in\mathcal{P}$. Then
	\begin{align*}
		\widehat{N_I}\widehat{N_J}
		&=\left[\vphantom{\displaystyle\sum_{j\in J}}\sum_{i\in I}\left(\widehat{H_i}-1\right)\right]\left[\sum_{j\in J}\left(\widehat{H_j}-1\right)\right]\\
		&=\sum_{(i,j)\in I\times J}\left(\widehat{H_i}-1\right)\left(\widehat{H_j}-1\right)\\
		&=\sum_{(i,j)\in I\times J}\left(\widehat{H_i}\widehat{H_j}-\widehat{H_i}-\widehat{H_j}+1\right)\\
		&=\sum_{(i,j)\in I\times J}\left[\widehat{G}-\left(\widehat{H_i}-1\right)-\left(\widehat{H_j}-1\right)+3\cdot 1\right]\\
		&=\norm{I}\norm{J}\widehat{G}-\norm{J}\widehat{N_I}-\norm{I}\widehat{N_J}+3\norm{I}\norm{J}\cdot1\\
		&=(\norm{I}-1)(\norm{J}-1)\widehat{G}+\norm{J}\left(\widehat{G}-\widehat{N_I}\right)+\norm{I}\left(\widehat{G}-\widehat{N_J}\right)+2\norm{I}\norm{J}\cdot1\\
		&=(\norm{I}-1)(\norm{J}-1)\widehat{G}+\norm{J}\sum_{\substack{L\in\mathcal{P}\\L\ne I}}\widehat{N_I}+\norm{I}\sum_{\substack{L\in\mathcal{P}\\L\ne J}}\widehat{N_J}+2\norm{I}\norm{J}\cdot1,
	\end{align*}
as required.
\end{proof}
We will call the supercharacter theory of Lemma~\ref{partitionsct} the {\it partition supercharacter theory of $G$ corresponding to $\mathcal{P}$}. As can be seen from the above proof, the reason that this construction works is because any two distinct normal subgroups generate $G$. As such, the same construction will work if $G$ is a direct product of two simple groups. That is, if $G=H_1\times H_2$, where $H_1\cong H_2$ is a simple group, then the set of nontrivial normal subgroups of $G$ is $\{H_1,H_2\}$. The only partitions of $\{1,2\}$ are $\{1,2\}$, which corresponds to $\mathrm{M}(G)$, and $\{\{1\},\{2\}\}$, which corresponds to $\mathrm{M}(H_1)\times\mathrm{M}(H_2)$. 

However, if $G$ has at least three nontrivial normal subgroups, it is not difficult to see that $G$ must be an elementary abelian group of rank two if any two distinct normal subgroups generate $G$. Indeed, this condition on $G$ implies that any nontrivial, proper normal subgroup is minimal normal. Let $L,N,M$ be distinct nontrivial, proper normal subgroups of $G$. Then $G=L\times N=L\times M=N\times M$, so $G=NM\le C_G(L)$ and $G=LM\le C_G(N)$. Thus we see that $G$ is abelian. So $L,N,M$ are cyclic of prime order. Since $L\cong G/N\cong M\cong G/L\cong N$, $G=C_p\times C_p$ for some prime $p$.

\begin{lem}\label{interval}
	Let $\mathcal{P}$ be the partition of $\{1,2,\dotsc,p+1\}$ consisting of all singletons. A supercharacter theory $\mathsf{S}$ is a partition supercharacter theory if and only if $\mathsf{S}_{\mathcal{P}}\preccurlyeq\mathsf{S}$. In particular the set of partition supercharacter theories is an interval in the lattice of supercharacter theories. 
\end{lem}

\begin{proof}
	The supercharacter theory $\mathsf{S}$ is a partition supercharacter theory if and only if $\inner{g}\setminus\{1\}\subseteq\cl_{\mathsf{S}}(g)$ holds for every $g\in G$. In other words, if and only if $g^i\in \cl_{\mathsf{S}}(g)$ holds for every $g\in G$ and $1\le i\le p-1$. This is exactly the condition $\mathsf{S}_{\mathcal{P}}\preccurlyeq\mathsf{S}$. Thus the interval $[\mathsf{S}_{\mathcal{P}},\mathsf{M}(G)]$ is the set of of partition supercharacter theories of $G$.
\end{proof}

Next we give an example that shows the set of automorphic supercharacter theories is not a semilattice of $\mathrm{SCT}(G)$. 

\begin{exam}
	Let $p\ge 5$. Let $H_1,H_2,H_3$ be distinct subgroups of $G$ of order $p$. Define the supercharacter theories $\mathsf{S}$ and $\mathsf{T}$ as follows: $\mathsf{S}=\mathsf{M}(H_1)\times\mathsf{M}(H_2)$, $\mathsf{T}=\mathsf{M}(H_2)\times\mathsf{M}(H_3)$, which are both partition supercharacter theories. So $\mathsf{S}\wedge\mathsf{T}$ is also a partition supercharacter theory by Lemma~\ref{interval}. It is not difficult to see that $H_1,H_2,H_3$ are the only nontrivial, proper $(\mathsf{S}\wedge\mathsf{T})$-normal subgroups of $G$. By Theorem~\ref{thmA}, both $\mathsf{S}$ and $\mathsf{T}$ come from automorphisms. However, $\mathsf{S}\wedge\mathsf{T}$ does not come from automorphisms, since $\mathsf{S}\wedge\mathsf{T}$ has exactly three nontrivial, proper supernormal subgroups (see Lemma~\ref{wielandt}). 
\end{exam} 

\section{Nontrivial $\mathsf{S}$-normal subgroups} \label{S normal}

In this section, we discuss the structure of the supercharacter theories of $G=C_p\times C_p$, $p$ a prime, that have nontrivial, proper supernormal subgroups. We begin by studying the structure of supercharacter theories with exactly one such subgroup. 

To prove Theorem~\ref{main actual}, we showed that if $[G,\mathsf{S}]$ or $Z(\mathsf{S})$ were the unique $\mathsf{S}$-normal subgroup of $G$, then $\mathsf{S}$ is a $\ast$-product. One may wonder if a similar result holds for any supercharacter theory of $G$ with a unique $\mathsf{S}$-normal subgroup. This is not the case, as the next example illustrates. 

\begin{exam}
We show that not every supercharacter theory of $C_p\times C_p$ with a unique nontrivial, proper supernormal subgroup is a $\ast$-product. This example comes from $G=C_5\times C_5$. Write $G=\inner{x,y}$. Let 
\begin{align*}K
	&=\bigl\{\{1\},\{y,y^2,y^3,y^4\},\{x,x^2,x^3,x^4,xy^4,x^2y^3,x^3y^2,x^4y\}\bigr\}\\
	&\cup\bigl\{\{xy,x^2y^2,x^3y^3,x^4y^4,xy^3,x^2y,x^3y^4,x^4y^2,xy^2,x^2y^4,x^3y,x^4y^3\}\bigr\}.
\end{align*}
One may readily verify that $K$ gives the set of $\mathsf{S}$-classes for a supercharacter theory $\mathsf{S}$ of $G$. Observe that $N=\inner{y}$ is the unique nontrivial, proper $\mathsf{S}$-normal subgroup. However $\mathsf{S}$ is not a $\ast$-product over $N$ since, for example, $x$ and $xy$ lie in different $\mathsf{S}$-classes.\hfill$\square$
\end{exam}

Observe that the above supercharacter theory is an example of a partition supercharacter theory. Indeed, the supercharacter theories $\mathsf{S}$ with a unique nontrivial, proper $\mathsf{S}$-normal subgroup that we have observed is either a $\ast$-product or a partition supercharacter theory. In particular, each supercharacter theory of this form has come from automorphisms or has been a partition supercharacter theory.

We have observed a similar phenomenon for those with exactly two nontrivial, proper $\mathsf{S}$-normal subgroups. Specifically, it appears as though every supercharacter theory of $G$ with exactly two nontrivial, proper supernormal subgroups either comes from automorphisms or is a partition supercharacter theory. It is however not the case that each such supercharacter theory that is not a partition supercharacter theory is a direct product, as can be seen from the next example.

\begin{exam}\label{not direct}
	Let $G=C_5\times C_5$. Let $x$ and $y$ be distinct nonidentity elements of $G$. Let $H=\inner{x}$ and $N=\inner{y}$. Let $\sigma$ be the automorphism of $G$ defined by $\sigma(x^iy^j)=x^{-i}y^{2j}$. Then the orbit of every element lying outside of $H$ has size four, and the orbit of every nonidentity element of $H$ has size two. Let $\mathsf{S}$ be the supercharacter theory of $G$ coming from $\inner{\sigma}$. Then $H$ and $N$ are the only nontrivial, proper $\mathsf{S}$-normal subgroups of $G$, $\norm{\mathsf{S}_H}=3$, $\norm{\mathsf{S}_N}=2$. Since $\norm{\mathsf{S}_H\times\mathsf{S}_N}=\norm{\mathsf{S}_H}\cdot\norm{\mathsf{S}_N}=6<8=\norm{\mathsf{S}}$, $\mathsf{S}$ is not a direct product by Corollary~\ref{direct}. 
\end{exam}

As mentioned just prior to Example~\ref{not direct}, it appears as though every supercharacter theory of $G$ with exactly two nontrivial, proper supernormal subgroups either comes from automorphisms or is a partition supercharacter theory (in fact, we have yet to find any supercharacter theory of $C_p\times C_p$ that does not either comes from automorphisms or partitions). We do have some evidence of this, but only have one weak result in this direction. Before giving the result, we set up some convenient notation that will be used for the remainder of the paper.

Let $\mathsf{S}$ be a supercharacter theory of $G$ and let $H$ be a nontrivial, proper $\mathsf{S}$-normal subgroup. Then $H$ is cyclic of prime order, so $\mathsf{S}_H$ comes from automorphisms, say from the the subgroup $A\le \mathrm{Aut}(H)$. Since $\mathrm{Aut}(H)\cong C_{p-1}$ is just the collection of power maps, there exists an integer $m$ such that $A$ is generated by the automorphism sending an element to its $m^{\rm th}$-power. Thus if $g\in H$, then $\cl_{\mathsf{S}_H}(g)=\{g,g^m,g^{m^2},\dotsc\}$. We denote this supercharacter theory by $[H]_m$. Given an integer $m$ and $g\in G$, we let $[g]_m$ denote the set $\{g,g^m,g^{m^2},\dotsc\}$. We let $\norm{m}_p$ denote the order of $m$ modulo $p$, which also the size of $[g]_m$ for $1\ne g\in G$. 

\begin{lem}\label{coprime 2}
	Let $\mathsf{S}$ have exactly two $\mathsf{S}$-normal subgroups $H$ and $N$. Write $\mathsf{S}_H=[H]_{m_1}$ and $\mathsf{S}_N=[N]_{m_2}$. Let $d_i=\norm{m_i}_p$, $i=1,2$, and assume that $(d_1,d_2)=1$. Then $\mathsf{S}=\mathsf{S}_H\times\mathsf{S}_N$. In particular, $\mathsf{S}$ comes from automorphisms.
\end{lem}

\begin{proof}
	Let $g\in G\setminus (H\cup N)$. Then $d_2=\norm{\cl_{\mathsf{S}_{G/H}}(g)}$ and $d_1=\norm{\cl_{\mathsf{S}_{G/N}}(g)}$. So $d_1d_2=\mathrm{lcm}(d_1,d_1)$ divides $\norm{\cl_{\mathsf{S}}(g)}$. Since $\mathsf{S}\preccurlyeq\mathsf{S}_H\times\mathsf{S}_N$, we know $\cl_{\mathsf{S}}(g)\subseteq\cl_{\mathsf{S}_H\times\mathsf{S}_N}(g)$. Since $d_1d_2\le\norm{\cl_{\mathsf{S}}(g)}\le \norm{\cl_{\mathsf{S}_H\times\mathsf{S}_N}(g)}=d_1d_2$, we conclude that $\cl_{\mathsf{S}}(g)=\cl_{\mathsf{S}_H\times\mathsf{S}_N}(g)$. Hence $\mathsf{S}=\mathsf{S}_H\times\mathsf{S}_N$, and the result follows from Lemma~\ref{autoproducts}.
\end{proof}

We suspect that the condition $(d_1,d_2)=1$ in Lemma~\ref{coprime 2} can be strengthened to $(d_1,d_2)<p-1$. This has only been totally verified for the primes $p=2,3$ and $5$. 

Now suppose that there are at least three $\mathsf{S}$-normal subgroups. We conjecture that $\mathsf{S}$ comes from automorphisms whenever the restriction to a $\mathsf{S}$-normal subgroup is not the coarsest theory. 

\begin{conj}\label{3 subs}
Let $\mathsf{S}$ be a supercharacter theory of $G$. Suppose that $G$ has at least three nontrivial, proper $\mathsf{S}$-normal subgroups, and let $H$ be one of them. If $\mathsf{S}_H\neq\mathsf{M}(H)$, then every subgroup of $G$ is $\mathsf{S}$-normal. In particular, $\mathsf{S}$ comes from automorphisms. 
\end{conj}  

Although a general proof appears to be difficult, this can be proved rather easily in a couple of specific cases.

\begin{lem}
	Let $\mathsf{S}$ be a supercharacter theory of $G$. Suppose that $G$ has at least three nontrivial, proper $\mathsf{S}$-normal subgroups. Let $H$ be $\mathsf{S}$-normal and write $\mathsf{S}_H=[H]_m$. If $\norm{m}_p=1$ or $2$, then Conjecture~$\ref{3 subs}$ holds.
\end{lem}

\begin{proof}
	First assume that $\norm{m}_p=1$, and let $K$ be another $\mathsf{S}$-normal subgroup. Then $\mathsf{S}_K=\mathsf{m}(K)$, so $G=\inner{H,K}\le Z(\mathsf{S})$.
	
	Now assume that $\norm{m}_p=2$. We may find distinct nonidentity elements $x,y\in G$such that $\inner{x}$, $\inner{y}$ and $\inner{xy}$ are all $\mathsf{S}$-normal. We show by induction on $n$ that $\inner{xy^n}$ is $\mathsf{S}$-normal for every $n\ge 1$. Let $[g]$ denote the set $\{g,g^{-1}\}$ for $g\in G$. Let $n\ge 2$ be the smallest integer for which $\inner{xy^n}$ is not $\mathsf{S}$-normal. Then $\inner{xy^{n-1}}$ is $\mathsf{S}$-normal, which means that $\widehat{[xy^{n-1}]}\widehat{[y]}$ can be expressed as a nonnegative integer linear combination of $\mathsf{S}$-class sums. Since $\widehat{[xy^{n-1}]}\widehat{[y]}=\widehat{[xy^n]}+\widehat{[xy^{n-2}]}$ and $\inner{xy^{n-2}}$ is $\mathsf{S}$-normal, $[xy^n]$ must be a union of $\mathsf{S}$-classes. So $\inner{[xy^n]}=\inner{xy^n}$ is also $\mathsf{S}$-normal, which contradicts the choice of $n$. Thus $\inner{xy^n}$ is $\mathsf{S}$-normal for every integer $n$, as claimed. 
\end{proof}

Conjecture~\ref{3 subs} also holds in the case that $\mathsf{S}$ comes from automorphisms. 

\begin{lem}\label{wielandt}
	Let $\mathsf{S}$ be a supercharacter theory of $G$ coming from automorphisms. If $G$ has at least three nontrivial, proper $\mathsf{S}$-normal subgroups, then every subgroup of $G$ is $\mathsf{S}$-normal.
\end{lem}

\begin{proof}
	Suppose that $\mathsf{S}$ comes from $A\le \mathrm{Aut}(G)$. Let $H_i$, $i=1,2,3$, be $\mathsf{S}$-normal of order $p$. We may assume that $H_1$ is generated by $x$, $H_2$ is generated by $y$ and $H_3$ is generated by $xy$. Let $a\in A$. Since $H_1$ and $H_2$ are $\mathsf{S}$-normal, there exist integers $i$ and $j$ so that $x^a=x^i$ and $y^a=y^j$. Then $(xy)^a=x^iy^j$. Since $H_3$ is $\mathsf{S}$-normal, $(xy)^a\in\inner{xy}$, which forces $i=j$. We conclude that $a\in Z(\mathrm{Aut}(G))$ and hence fixes every subgroup of $G$.
\end{proof}

We remark that Lemma~\ref{wielandt} also follows from \cite[Lemma 26.3]{HW64}, a result about Schur rings.

We now outline a strategy we believe will work to prove Conjecture~\ref{3 subs}, although the actual proof has evaded us. This strategy involves an algorithm of the first author appearing in \cite{sct meets}, which we now describe. We begin by defining an equivalence relation on $G$ coming from a partial partition of $G$. Let $\mathcal{C}$ be a $G$-invariant partial partition of $G$. For $K,L\in\mathcal{C}$ and $g\in G$, define
\[(K,L)_g=\norm{\{(k,l)\in K\times L:kl=g\}}.\]Also recall that for $K\subseteq G$, $\hat{K}$ denotes the element $\sum_{g\in K}g$ of the group algebra $\mathbb{Z}(G)$. 

Define an equivalence relation $\sim$ on $G$ by defining $g\sim h$ if and only if 
\[(K,L)_g=(K,L)_h\]
for all $K,L\in\mathcal{C}$. Let $\mathscr{K}(\mathcal{C})$ denote the set of equivalence classes of $\mathrm{Irr}(G)$ under $\sim$. It is shown in \cite{sct meets} that $\mathscr{K}(\mathcal{C})$ is a refinement of $\mathcal{C}$ and that $\mathscr{K}(\mathcal{C})=\mathcal{C}$ if and only if $\mathcal{C}$ is the set of superclasses for a supercharacter theory $\mathsf{S}$ of $G$. If the partition $\mathrm{Cl}(\mathsf{S})$ of $G$ corresponding to the supercharacter theory $\mathsf{S}$ of $G$ is a refinement of $\mathcal{C}$, then $\mathrm{Cl}(\mathsf{S})$ is a refinement of $\mathscr{K}(\mathcal{C})$. Iterating on this process, the sequence $\mathcal{C}, \mathscr{K}(\mathcal{C}), \mathscr{K}^2(\mathcal{C}),\dotsc$ eventually terminates; call its terminal member $\mathscr{K}^\infty(\mathcal{C})$. Then $\mathscr{K}^\infty(\mathcal{C})$ is the set of superclasses for a supercharacter theory $\mathsf{T}$ of $G$, and $\mathsf{S}\preccurlyeq\mathsf{T}$.

Let $\mathsf{S}$ be a supercharacter theory of $G$, and assume that $H_1,H_2,H_3$ are distinct nontrivial, proper $\mathsf{S}$-normal subgroups. Assume further that $\mathsf{S}_{H_1}\ne\mathsf{M}(H_1)$. Let $r$ be a primitive root modulo $p$. If $\mathsf{S}_{H_1}=[H_1]_m$, then $[H_1]_m\preccurlyeq[H_1]_{r^2}$, since $\mathsf{S}_{H_1}\ne\mathsf{M}(H_1)$. In particular, $\mathrm{Cl}(\mathsf{S})$ is a refinement of the partition $\bigcup_{i=1}^3[H_i]_m\cup\{G\setminus\bigcup_{i=1}^3H_i\}$, which is a refinement of the partition $\mathcal{C}=\bigcup_{i=1}^3[H_i]_{r^2}\cup\{G\setminus\bigcup_{i=1}^3H_i\}$. If $g\in G\setminus\bigcup_{i=1}^nH_i$ such that $[g]_{r^2}\in\mathscr{K}^\infty(\mathcal{C})$, then $[g]_{r^2}$ must be a union of $\mathsf{S}$-classes. Thus $\inner{[g]_{r^2}}=\inner{g}$ is another $\mathsf{S}$-normal subgroup. 

\begin{conj}\label{reduction}
Let $H_1,H_2$ and $H_3$ be distinct subgroups of $G$ of order $p$. Let $r$ be a primitive root modulo $p$. Define $\mathcal{C}=\bigcup_{i=1}^3[H_i]_{r^2}\cup\{G\setminus\bigcup_{i=1}^3H_i\}$. For every $g\in G\setminus\bigcup_iH_i$, the set $[g]_{r^2}$ lies in $\mathscr{K}^\infty(\mathcal{C})$.
\end{conj} 

We now illustrate a couple of small examples illustrating the statement of Conjecture~\ref{reduction}.

\begin{exam}
Let $G=\inner{x,y}$ be the group $C_5\times C_5$. Suppose that $\mathsf{S}$ is a supercharacter theory of $G$ with at least three supernormal subgroups $H_1$, $H_2$ and $H_3$. Relabeling if necessary, we may write $H_1=\inner{x}$, $H_2=\inner{y}$ and $H_3=\inner{xy}$. If $\mathsf{S}_{H_1}\ne\mathsf{M}(H_1)$, then $\mathsf{S}_{H_1}$ comes from the inversion automorphism. So $[x]_4$ and $[y]_4$ are $\mathsf{S}$-classes. So 
\[\widehat{[x]_4}\widehat{[y]_4}=\widehat{[xy]_4}+\widehat{[xy^{-1}]_4}\]
is a sum of $\mathsf{S}$-class sums. Since $\inner{xy}$ is $\mathsf{S}$-normal, it follows that $L_1=[xy^{-1}]_4$ is a union of $\mathsf{S}$-classes. Thus so is $L_1^i=[x^iy^{-i}]_4$ for each $2\le i\le p-1$. Since $\inner{xy^{-1}}=\bigcup_iL_1^i$, we deduce that $\inner{xy^{-1}}$ is also $\mathsf{S}$-normal.

Similarly, $[y]_4$ and $[xy]_4$ are $\mathsf{S}$-classes. Thus
\[\widehat{[y]_4}\widehat{[xy]_4}=\widehat{[x]_4}+\widehat{[xy^2]_4}\]
is a union of $\mathsf{S}$-classes. Since $\inner{x}$ is $\mathsf{S}$-normal, $[xy^2]_4^i=\{x^iy^{2i},x^{-i}y^{-2i}\}$ is a union of $\mathsf{S}$-classes for each $2\le i\le p-1$. From this, we conclude that $\inner{xy^2}$ is also $\mathsf{S}$-normal.

Using a similar technique with the classes $[y^2]_4$ and $[xy]_4$, we can also show that $\inner{xy^3}$ is $\mathsf{S}$-normal. \hfill$\square$
\end{exam}

\begin{exam}
	Let $G=\inner{x,y}$ be the group $C_7\times C_7$. Suppose that $\mathsf{S}$ is a supercharacter theory of $G$ with at least three supernormal subgroups. As with the previous example, we may assume $H_1=\inner{x}$, $H_2=\inner{y}$ and $H_3=\inner{xy}$ are $\mathsf{S}$-normal. Suppose that $\mathsf{S}_{H_1}$ comes from the squaring automorphism. Then $[x]_2=\{x,x^2,x^4\}$ and $[y]_2=\{y,y^2,y^4\}$ are $\mathsf{S}$-classes. So
	\[\widehat{[x]_2}\widehat{[y]_2}=\widehat{[xy]_2}+\widehat{[xy^2]_2}+\widehat{[xy^4]_2}
	\]
is a sum of $\mathsf{S}$-classes. Since $\inner{xy}$ is $\mathsf{S}$-normal, $L_1=[xy^2]_2\cup[xy^4]_2$ is a union of $\mathsf{S}$-classes. Similarly, using the $\mathsf{S}$-classes of $y^{-1}$ and $xy$, we see that $L_2=[xy^4]_2\cup[xy^{-1}]_2$ is a union of $\mathsf{S}$-classes. So $L=L_1\cap L_2=[xy^4]_2$ is a union of $\mathsf{S}$-classes. Thus we may conclude that $\inner{xy^4}, \inner{xy^2}$ and $\inner{xy^{-1}}$ are also $\mathsf{S}$-normal.

One may easily verify that a similar process shows that every subgroup of $G$ must be $\mathsf{S}$-normal.\hfill$\square$
\end{exam}

\begin{exam}
	Let $G=\inner{x,y}$ be the group $C_{11}\times C_{11}$. Suppose that $\mathsf{S}$ is a supercharacter theory of $G$ with at least three supernormal subgroups $H_1,H_2$ and $H_3$, and suppose that $\mathsf{S}_{H_1}=[H_1]_3$. As with the previous example, we may assume $H_1=\inner{x}$, $H_2=\inner{y}$ and $H_3=\inner{xy}$. Let $\mathcal{C}$ be the partition 
	\[\mathcal{C}=\{\{1\},[x]_3,[x^2]_3,[y]_3,[y^2]_3,[xy]_3,[x^2y^2]_3,G\setminus(\inner{x}\cup\inner{y}\cup\inner{xy})\}\]
	 of $G$. Note that $\mathsf{S}$ is a refinement of $\mathcal{C}$. It is routine to show that
	 \begin{align*}
	 	\mathscr{K}(\mathcal{C})=\{&\{1\},[x]_3,[x^2]_3,[y]_3,[y^2]_3,[xy]_3,[x^2y^2]_3,[xy^3]_3\cup[xy^9]_3\}\\
		&\cup \{[xy^4]_3\cup[xy^5]_3,[xy^7]_3\cup[xy^8]_3, [x^2y^3]_3\cup[x^2y^5]_3,[x^2y^6]_3\cup[x^2y^7]_3\}\\
		&\cup \{[x^2y^8]_3\cup[x^2y^{10}]_3,[xy^2]_3\cup [xy^6]_3\cup [xy^{10}]_3, [x^2y]_3\cup [x^2y^4]_3\cup [x^2y^9]_3\}
	\end{align*}
	and that
	\begin{align*}
		(\widehat{[xy^3]_3}+\widehat{[xy^9]_3})(\widehat{[xy^4]_3}+\widehat{[xy^5]_3})&=
		\widehat{[xy]_3}+3\widehat{[x^2y^2]_3}+\bigl(\widehat{[x^2y^6]_3}+\widehat{[x^2y^7]_3}\bigr)\\
		&+\bigl(\widehat{[x^2y^8]_3}+\widehat{[x^2y^{10}]_3}\bigr)+\widehat{[xy^5]_3}+2\widehat{[xy^6]_3}\\
		&+2\widehat{[xy^8]_3}+\widehat{[xy^9]_3}+\widehat{[xy^{10}]_3}+2\widehat{[x^2y^3]_3}\\
		&+2\widehat{[x^2y^4]_3}+\widehat{[x^2y^9]_3}.
	\end{align*}
	
	Since $\widehat{[xy^4]_3}$ does not appear in the above decomposition and $\mathsf{S}$ is a refinement of $\mathscr{K}^2(\mathcal{C})$, we conclude that $\inner{xy^4}$ and $\inner{xy^5}$ are $\mathsf{S}$-normal. Similary, $\inner{xy^7}$ and $\inner{xy^8}$ are $\mathsf{S}$-normal. The same reasoning also shows that $\inner{xy^3}$ and $\inner{xy^9}$ are $\mathsf{S}$-normal.
	
	Since $\widehat{[xy^2]_3}$ does not appear in the above decomposition, and since $\widehat{[xy^6]_3}$ and $\widehat{[xy^{10}]_3}$ appear with different coefficients, we see that $\inner{xy^2}$, $\inner{xy^6}$ and $\inner{xy^{10}}$ are also $\mathsf{S}$-normal. Thus every subgroup of $G$ is $\mathsf{S}$-normal.
\end{exam}

We have verified Conjecture~\ref{reduction} for every prime $p\le 47$ using the method outlined above, thereby also verifying Conjecture~\ref{3 subs} in the process. In the event that Conjecture~\ref{3 subs} holds, we can conclude that a large collection of supercharacter theories of $G$ must come from automorphisms.

\begin{lem}\label{autos}
Let $\mathsf{S}$ be a supercharacter theory of $G$. If every subgroup of $G$ is $\mathsf{S}$-normal, then $\mathsf{S}$ comes from automorphisms. 
\end{lem}

\begin{proof}
Let $H_1$, $H_2$ and $H_3$ be distinct subgroups of $G$. Suppose that $\mathsf{S}_{H_1}$ comes from the automorphism that sends an element to its $m^{\rm th}$ power. Let $\phi_{12}:H_1\to G/H_2$, $\phi_{13}:H_1\to G/H_3$ and $\phi_{23}:H_2\to G/H_3$ be the canonical isomorphisms. The $\mathsf{S}_{G/H_2}$-classes are the images of the $\mathsf{S}_{H_1}$-classes under $\phi_{12}$ and the $\mathsf{S}_{G/H_3}$-classes are the images of the $\mathsf{S}_{H_2}$-classes under $\phi_{23}$. So the $\mathsf{S}_{H_2}$-classes are the images of the $\mathsf{S}_{H_1}$-classes under $\phi_{23}^{-1}\phi_{12}$. Thus $\mathsf{S}_{H_2}$ also comes from the automorphism that sends an element to its $m^{\rm th}$ power. Since $H_2$ was chosen randomly, it follows that the $\mathsf{S}$-class of an element $g\in G$ is just $\{g,g^m,g^{m^2},\dotsc\}$. So $\mathsf{S}$ comes from automorphisms, as claimed.
\end{proof}

We have observed in small cases ($p=2,3,5$) that every supercharacter theory either comes from automorphisms or is a partition supercharacter theory. Due the large number of conjugacy classes in $C_7\times C_7$, it is unfortunately rather difficult computationally to verify the observation in this case. It is known (e.g. by \cite[Theorem 2.2 (f)]{ID07}) that if $e$ is an integer coprime to $p$, then $K^e=\{g^e:g\in G\}$ is an $\mathsf{S}$-class for every $\mathsf{S}$-class $K$. In every example we have computed, a very special choice of $e$ actually fixes every $\mathsf{S}$-class as long as $G$ has at least three nontrivial, proper $\mathsf{S}$-normal subgroups. Thus the following conjecture seems reasonable, and would allow us to greatly reduce the possible structure of supercharacter theories of $G$.

\begin{conj}\label{coarse theories}
	Let $\mathsf{S}$ be a supercharacter theory of $G$ with at least three nontrivial, proper $\mathsf{S}$-normal subgroups. Let $H$ be a nontrivial $\mathsf{S}$-normal subgroup and write $\mathsf{S}_H=[H]_m$. Let $\mathsf{T}$ be the supercharacter theory of $G$ coming from the automorphism sending an element to its $m^{\rm th}$ power. Then $\mathsf{T}\preccurlyeq\mathsf{S}$. 
\end{conj}

\begin{prop}\label{conj prop}
	Assume Conjecture~$\ref{coarse theories}$ holds. Let $\mathsf{S}$ be a supercharacter theory of $G$ with at least three nontrivial, proper $\mathsf{S}$-normal subgroups. Suppose that $\mathsf{S}_H=\mathsf{M}(H)$ for every $\mathsf{S}$-normal subgroup $H$ of $G$. Then $\mathsf{S}$ is a partition supercharacter theory.
\end{prop}

\begin{proof}
	Let $H_1,H_2,\dotsc, H_n$ be the $\mathsf{S}$-normal subgroups of $G$ of order $p$. Let $g\in G\setminus\bigcup_iH_i$. Let $r$ be a primitive root modulo $p$ and note that $\mathsf{S}_{H_i}$ comes from the automorphism sending an element to its $r^{\rm th}$ power for each $1\le i\le n$. By hypothesis Conjecture~\ref{coarse theories} holds, so $\mathsf{T}\preccurlyeq\mathsf{S}$ where $\mathsf{T}$ is the supercharacter theory of $G$ coming from the automorphism sending an element to its $r^{\rm th}$ power. Observe that $\mathsf{T}$ is the partition supercharacter theory $\mathsf{S}_{\mathcal{P}}$ corresponding to the partition $\mathcal{P}$ consisting of all singletons. The result now follows from Lemma~\ref{interval}.
\end{proof}

\section{Computations for small primes}

In this section, we consider the cases $p=2,3,5,7$ and $11$ in some depth. For $p=2,3$ and $5$, we can fully classify the supercharacter theories of $C_p\times C_p$. We are able to fully classify the supercharacter theories of $C_7\times C_7$ that have at least three nontrivial, proper supernormal subgroups. This allows us to verify Conjecture~\ref{coarse theories} for the case $p=7$. 

As mentioned earlier, we have verified Conjecture~\ref{3 subs} for all primes $p\le 47$. So we can classify the supercharacter theories of $C_{11}\times C_{11}$ that have at least three nontrivial, proper supernormal subgroups and for which the restriction to one of them is not the coarsest theory. We do a little more for $p=11$. We classify all supercharacter theories of $C_{11}\times C_{11}$ that have at least nine nontrivial, proper supernormal subgroups.

We will say that a supercharacter theory $\mathsf{S}$ has type $T_n$ if $G$ has exactly $n$ nontrivial, proper $\mathsf{S}$-normal subgroups. In each of the cases just mentioned, we list the total number of supercharacter theories of type $T_n$. For each $n$, we divide the supercharacter theories of type $T_n$ into three major types: (1) Those that can be realized as a nontrivial $\ast$ or direct product; (2) Those coming from automorphisms (automorphic supercharacter theories); (3) Those coming from partitions. We give exact counts for each type as well.  

Before discussing the computational aspect, let us find the number of $\ast$ and direct products. Since every supercharacter theory of $C_p$ comes from automorphisms, and two distinct subgroups of $\mathrm{Aut}(C_p)$ produce two distinct supercharacter theories, we find that the number of distinct supercharacter theories of $C_p$ is $\tau(p-1)$, where $\tau(n)$ is the number of divisors of the integer $n$. A $\ast$-product of $G$ is determined by three pieces of information: A nontrivial, proper (normal) subgroup $N$, a supercharacter theory of $N$, and a supercharacter theory of $G/N$. Since $N\cong G/N$, it follows that  there are $(p+1)\tau(p-1)^2$ supercharacter theories of $G=C_p\times C_p$ that can be realized as a nontrivial $\ast$-products. Note that $\mathsf{m}(G)=\mathrm{m}(H)\times\mathsf{m}(N)$ for any choice of two distinct nontrivial, proper subgroups $H$ and $N$. Thus there are $1+\binom{p+1}{2}(\tau(p-1)^2-1)$ supercharacter theories of $G=C_p\times C_p$ that can be realized as nontrivial direct products, where we are including $\mathsf{m}(G)$.

To compute the supercharacter theories of $G$ coming from automorphisms, we constructed the natural action of $\mathrm{Aut}(G)$ on $\irr(G)$, which gave us a faithful representation of $\mathrm{Aut}(G)$ into $\mathrm{Sym}(\irr(G))$. Using this permutation representation, we used MAGMA's \verb+Subgroups+ function to find representatives of the conjugacy classes of subgroups of $\mathrm{Aut}(G)$. Finally, we expanded the classes to find all of the subgroups of $\mathrm{Aut}(G)$ and constructed the orbits of these subgroups on $\irr(G)$.

Computing the total number of partition supercharacter theories of each type is an easy combinatorics exercise. The number of supernormal subgroups of $\mathsf{S}_{\mathcal{P}}$ is the multiplicity of 1 in the partition $\mathcal{P}$. Suppose $\mathcal{P}$ has shape $1^{n_1}+2^{n_2}+\dotsb+(p+1)^{n_{p+1}}$. Let $m_0=0$ and define $m_i=m_{i-1}+in_i$ for $i\ge 1$. Let 
\[\mathcal{N}_i= \frac{\binom{p+1-m_{i-1}}{i}\binom{p+1-m_{i-1}-n_i}{i}\dotsb\binom{p+1-m_{i-1}-i(n_i-1)}{i}}{n_i!}.\]
Then the number of partition supercharacter theories with shape $\mathcal{P}$ is $\prod_{i=1}^{p+1}\mathcal{N}_i$.

We remind the reader that all of the product supercharacter theories are automorphic supercharacter theories.  Hence, those two entries are equal for a given group if and only if all of the automorphic supercharacter theories are product supercharacter theories.  Also, it is possible for a supercharacter theory to be both an automorphic supercharacter theory and a partition supercharacter theory.  The overlap between these two sets explains why the total number of supercharacter theories is less that then the sum of the number of partition supercharacter theories and automorphic supercharacter theories.

If $p=2$, then one can easily compute $\mathrm{SCT}(G)$ by hand. In fact, it would not be terribly difficult to do so for $p=3$. However, we used Hendrickson's algorithm given in \cite{AH08} to compute $\mathrm{SCT}(G)$ for $p=3$ and $p=5$. This is unfortunately computationally infeasible for $p>5$.

Now let $p=7$ or $p=11$. For $i$ sufficiently large, we were able to compute the total number of supercharacter theories of type $T_i$. To accomplish this, we showed that every supercharacter theory of type $T_i$ either came from automorphisms or partitions. To accomplish this, we used the following sequence of steps.

\noindent{\bf Step 1:} Suppose that $\mathsf{S}$ is a supercharacter theory of $G=C_p\times C_p$ of type $T_i$ for $i\ge 3$ and $p\in\{7,11\}$. Let $H$ be an $\mathsf{S}$-normal subgroup of $G$. Since we have checked Conjecture 6.4 for $G$, we know that $\mathsf{S}_H=\mathsf{M}(H)$. Since $G$ has at least three nontrivial, proper $\mathsf{S}$-normal subgroups, we also know that $\mathsf{S}_{G/H}=\mathsf{M}(G/H)$. Let $g$ be an element of $G$ that does not lie in an $\mathsf{S}$-normal subgroup of $G$. Since $\norm{\cl_{\mathsf{S}_{G/H}}(g)}=p-1$ and $\norm{\cl_{\mathsf{S}_{G/H}}(g)}$ divides $\norm{\cl_{\mathsf{S}}(g)}$, we know that $\norm{\cl_{\mathsf{S}}(g)}$ is divisible by $p-1$.

\noindent{\bf Step 2:} Let $H_1,\dotsc,H_n$ be the set of all nontrivial, proper $\mathsf{S}$-normal subgroups of $G$. Let $L=G\setminus\bigcup_{i=1}^nH_i$. Let $\mathcal{G}$ be the subset of all $P\subseteq L$ for which $\norm{P}$ is divisible by $p-1$ and also satisfy $\{P\}\in\mathcal{K}^\infty(\{P\})$ (Hendrickson calls subsets satisfying this latter condition {\it good} subsets).

\noindent{\bf Step 3:} For each $P\in\mathcal{G}$, determine if $\mathcal{C}_P\subseteq \mathscr{K}^\infty(\mathcal{C}_P)$, where $\mathcal{C}_P=\{\{1\},H_1\setminus\{1\},H_2\setminus\{1\},\dotsc,H_n\setminus\{1\},P\}$. If $P$ does not satisfy this condition, then $P$ is not an $\mathsf{S}$-class. After completing each of the above steps, the only subsets $P$ that remained had the form $P=\bigcup_{x\in J}(\inner{x}-1)$ for some subset $J\subseteq L$. In particular, $\mathsf{S}$ must be a partition supercharacter theory by Proposition~\ref{conj prop}.

The table summarizes our computational results. We indicate the information that is unavailable with a dash. The only such information is the total number of supercharacter theories of $G=C_7\times C_7$ of type $T_i$ for $0\le i\le 2$ and the total number of supercharacter theories of $G=C_{11}\times C_{11}$ of type $T_i$ for $0\le i\le 8$. 

\begin{center}
\begin{tabular}{|c|c|c|c|c|c|}
	\hline
		Prime&Type&Product&Automorphic&Partition&Total\\
	\hline\hline
		\multirow{3}{*}{$2$}&$T_0$&0&1&1&1 \\
	\cline{2-6}
		&$T_1$&3&3&3&3 \\
	\cline{2-6}
		&$T_3$&1&1&1&1 \\
	\hline\hline
		\multirow{4}{*}{$3$}&$T_0$&0&4&4&4 \\
	\cline{2-6}
		&$T_1$&16&16&4&16 \\
	\cline{2-6}
		&$T_2$&18&18&6&18 \\
	\cline{2-6}
		&$T_4$&1&2&1&2 \\
%
	\hline\hline
		\multirow{6}{*}{$5$}&$T_0$&0&96&41&96 \\
	\cline{2-6}
		&$T_1$&54&54&66&114 \\
	\cline{2-6}
		&$T_2$&120&180&60&210 \\
	\cline{2-6}
		&$T_3$&0&0&20&20 \\
	\cline{2-6}
		&$T_4$&0&0&15&15 \\
	\cline{2-6}
		&$T_6$&1&3&1&3 \\
	\hline\hline
		\multirow{8}{*}{$7$}&$T_0$&0&470&715&- \\
	\cline{2-6}
		&$T_1$&128&128&1296&- \\
	\cline{2-6}
		&$T_2$&420&728&1148&- \\
	\cline{2-6}
		&$T_3$&0&0&616&616 \\
	\cline{2-6}
		&$T_4$&0&0&280&280 \\
	\cline{2-6}
		&$T_5$&0&0&56&56 \\
	\cline{2-6}
		&$T_6$&0&0&28&28 \\
	\cline{2-6}
		&$T_8$&1&4&1&4 \\
	\hline\hline
		\multirow{12}{*}{$11$}&$T_0$&0&2839&580317&- \\
	\cline{2-6}
		&$T_1$&192&192&1179036&- \\
	\cline{2-6}
		&$T_2$&990&2376&1169652&- \\
	\cline{2-6}
		&$T_3$&0&0&753500&- \\
	\cline{2-6}
		&$T_4$&0&0&353925&- \\
	\cline{2-6}
		&$T_5$&0&0&128304&- \\
	\cline{2-6}
		&$T_6$&0&0&37884&- \\
	\cline{2-6}
		&$T_7$&0&0&8712&- \\
	\cline{2-6}
		&$T_8$&0&0&1980&- \\		
	\cline{2-6}
		&$T_9$&0&0&220&220 \\		
	\cline{2-6}
		&$T_{10}$&0&0&66&66 \\		
	\cline{2-6}
		&$T_{12}$&1&4&1&4 \\		
	\hline
\end{tabular}
\end{center}

Given all of the evidence presented thus far, it is feasible that one may be able to classify the supercharacter theories of $C_p\times C_p$ that have at least one nontrivial, proper $\mathsf{S}$-normal subgroup (a supercharacter theory with no nontrivial, proper $\mathsf{S}$-normal subgroup has been referred to as {\it simple}). The others may prove much more difficult to describe. For example, if $N\lhd_{\mathsf{S}}G$, then $\mathsf{S}\preccurlyeq\mathsf{S}_N\ast\mathsf{S}_{G/N}$, which puts limitations on which elements can lie in the same $\mathsf{S}$-class. Also $\norm{\cl_{\mathsf{S}_{G/N}}(gN)}$ must divide $\norm{\cl_{\mathsf{S}}(g)}$ for every $g\in G\setminus N$, which puts arithmetic restrictions on the possible sizes of $\mathsf{S}$-classes. When $\mathsf{S}$ is simple though, none of these restrictions apply. However, given all that we have observed, it is possible that a classification may only involve the partition supercharacter theories and those coming from automorphisms.

%


\end{document}